\documentclass[12pt]{article}
\usepackage{amsmath, amssymb, amsfonts}
\usepackage{amsthm}
\usepackage[utf8]{inputenc}
\usepackage{csquotes}
\usepackage{geometry}
\usepackage{adjustbox}
\usepackage{graphicx}
\usepackage{verbatim}
\usepackage{array}
\usepackage{enumitem}
\usepackage{booktabs}
\usepackage{hyperref}
\usepackage{url}
\usepackage[capitalize]{cleveref}
\usepackage{tikz}
\usetikzlibrary{positioning, shapes.geometric, arrows.meta}

\title{Smooth Integer Encoding via Integral Balance}
\author{Stanislav Semenov \\
\href{mailto:stas.semenov@gmail.com}{stas.semenov@gmail.com} \\
\href{https://orcid.org/0000-0002-5891-8119}{ORCID: 0000-0002-5891-8119}}
\date{April 28, 2025}

\theoremstyle{definition}
\newtheorem{definition}{Definition}[section]

\theoremstyle{plain}

\newtheorem{theorem}[definition]{Theorem}

\newtheorem{corollary}[definition]{Corollary}
\newtheorem{proposition}[definition]{Proposition}

\theoremstyle{remark}

\begin{document}

\maketitle

\begin{abstract}
  We introduce a novel method for encoding integers using smooth real-valued functions whose integral properties implicitly reflect discrete quantities.  
  In contrast to classical representations, where the integer appears as an explicit parameter, our approach encodes the number \( N \in \mathbb{N} \) through the cumulative balance of a smooth function \( f_N(t) \), constructed from localized Gaussian bumps with alternating and decaying coefficients.  
  The total integral \( I(N) \) converges to zero as \( N \to \infty \), and the integer can be recovered as the minimal point of near-cancellation.
  
  This method enables continuous and differentiable representations of discrete states, supports recovery through spline-based or analytical inversion, and extends naturally to multidimensional tuples \( (N_1, N_2, \dots) \).  
  We analyze the structure and convergence of the encoding series, demonstrate numerical construction of the integral map \( I(N) \), and develop procedures for integer recovery via numerical inversion.  
  The resulting framework opens a path toward embedding discrete logic within continuous optimization pipelines, machine learning architectures, and smooth symbolic computation.
  \end{abstract}

\subsection*{Mathematics Subject Classification}
03F60 (Constructive and recursive analysis), 26E40 (Constructive analysis)

\subsection*{ACM Classification}
F.4.1 Mathematical Logic, G.1.0 Numerical Analysis

\section{Introduction}

Representing discrete quantities such as integers within continuous mathematical frameworks is a central challenge in optimization, numerical analysis, and machine learning. Traditional symbolic representations and modern soft relaxation techniques both face fundamental limitations: the former lack differentiability, while the latter introduce approximation errors and auxiliary complexities.

In this work, we propose a novel method for encoding integers through smooth real-valued functions whose integral properties implicitly reflect discrete quantities. Rather than exposing the integer \( N \in \mathbb{N} \) as a parameter, we construct a smooth function \( f_N(t) \) whose cumulative integral encodes \( N \) as a point of near-cancellation:
\[
I(N) := \int_{-\infty}^{\infty} f_N(t)\,dt.
\]
The construction employs localized Gaussian bumps with alternating-decaying coefficients to ensure convergence, smoothness, and functional emergence of discrete states.

This approach enables:
\begin{itemize}
  \item smooth and differentiable representations of integers,
  \item principled recovery through numerical inversion,
  \item and natural extension to multidimensional integer data.
\end{itemize}

Detailed motivations, comparisons, and structural properties of the method are discussed in the next sections.

\section{Motivation and Comparison with Existing Methods}

\subsection{Classical Discrete Representations}

Encoding discrete quantities such as integers within continuous mathematical frameworks is a fundamental challenge across various domains, including numerical analysis, optimization, differentiable programming, and machine learning. Traditional approaches rely on explicit symbolic representations, where integers are treated as direct parameters or variables. 

While effective in purely discrete settings, symbolic methods are inherently non-differentiable and thus poorly suited for integration into continuous optimization pipelines or differentiable architectures~\cite{baydin2018automatic}. Consequently, they limit the applicability of gradient-based techniques and often require combinatorial or discrete search strategies.

\subsection{Limitations of Soft Quantization and Differentiable Relaxations}

To address the incompatibility between discrete variables and differentiable systems, various relaxation techniques have been developed. These include soft quantization, surrogate loss functions, temperature-based sampling (e.g., Gumbel-Softmax~\cite{jang2016gumbel}), and continuous approximations of discrete sets.

However, such methods typically introduce several drawbacks:
\begin{itemize}
  \item \textbf{Loss of Precision:} Soft relaxations approximate discrete states only approximately, resulting in errors that can propagate or accumulate during optimization.
  \item \textbf{Non-Differentiability at Limits:} Despite smooth approximations within finite regions, exact recovery of discrete quantities often remains non-differentiable or discontinuous at critical thresholds.
  \item \textbf{Auxiliary Parameters:} Many relaxations require tuning additional parameters, such as temperature schedules, which complicate the optimization landscape.
\end{itemize}

Thus, while soft methods provide partial solutions, they do not offer a fully principled or mathematically grounded way to encode discrete structure within continuous domains.

\subsection{Advantages of Integral-Based Encoding}

In this work, we propose a novel method for representing integers not as explicit parameters, but as emergent properties of a smooth functional system. Specifically, we construct smooth real-valued functions \( f_N(t) \), where the integer \( N \in \mathbb{N} \) is not directly exposed but can be recovered via a global measurable property: the cancellation of the total integral.

The integral map is defined by
\[
I(N) := \int_{-\infty}^{\infty} f_N(t)\,dt,
\]
where \( f_N(t) \) is composed of localized smooth Gaussian bumps with carefully structured coefficients.

\begin{definition}
For each integer \( N \geq 1 \), the smooth encoding function \( f_N(t) \) is defined by
\[
f_N(t) := \sum_{n=1}^{N} a_n \exp\left( -\frac{(t-n)^2}{2\delta^2} \right),
\]
where \( a_n \) are prescribed coefficients and \( \delta > 0 \) controls the smoothness and localization of each Gaussian bump.
\end{definition}

\paragraph{Nature of \( f_N(t) \) versus \( I(N) \).}
It is important to distinguish the different roles of the variables \( t \) and \( N \) in the constructions.  
The function \( f_N(t) \) is a smooth function of the continuous variable \( t \in \mathbb{R} \), with \( N \) serving as a hyperparameter that controls the number of localized Gaussian bumps included in the sum.  
In contrast, the integral map \( I(N) \) treats \( N \) as the independent variable: it records the cumulative integral of \( f_N(t) \) as a function of \( N \).  
Thus, while \( f_N(t) \) is evaluated over a continuous domain in \( t \), the function \( I(N) \) is defined over discrete or continuous values of \( N \), reflecting the accumulated contribution up to the \( N \)-th bump.

The key innovations of the proposed method include:
\begin{itemize}
  \item \textbf{Smooth and Differentiable Representation:} The function \( f_N(t) \) is infinitely differentiable, enabling seamless integration into gradient-based frameworks.
  \item \textbf{Principled Recovery Mechanism:} The integer \( N \) emerges as the minimal balance point of the integral map \( I(N) \), enabling efficient numerical inversion through spline interpolation or root-finding methods.
  \item \textbf{Generalization to Multidimensional Structures:} The encoding naturally extends to tuples \( (N_1, N_2, \dots) \) via independent or coupled integral balances across coordinate functions.
\end{itemize}

\paragraph{Why Not Arbitrary Oscillations?}

One might ask whether any oscillatory or sign-alternating function could serve a similar purpose: encoding an integer implicitly through the behavior of its integral. Examples might include trigonometric functions such as \( \sin(nt) \), \( \cos(\pi t^2) \), or general waveforms.

However, generic oscillatory functions suffer from several fundamental limitations:
\begin{itemize}
  \item \textbf{Lack of Controlled Cancellation:} Oscillations may not decay in magnitude, leading to divergent or unstable global integrals that do not encode a well-defined cumulative state.
  \item \textbf{Non-Convergent Behavior:} Without structured decay, the global integral may fail to converge, making recovery of discrete quantities unreliable or impossible.
  \item \textbf{Unstructured Oscillations:} Even if local cancellation occurs, the absence of systematic decay or organization prevents robust identification of unique balance points.
\end{itemize}

\paragraph{Structured Decay and Convergence.}

In contrast, our method is based on the smooth realization of a carefully constructed alternating-convergent series. The coefficients
\[
a_n := \frac{(1/2)^n + (-1)^n}{n}
\]
are designed to exhibit:
\begin{itemize}
  \item \textbf{Geometric decay}, ensuring convergence and stabilization of the global integral;
  \item \textbf{Alternating signs}, inducing partial cancellations that create identifiable balance points.
\end{itemize}
This structure leads to partial sums that oscillate in a controlled manner and converge reliably toward zero, with the integer \( N \) corresponding to the location of near-cancellation.

\paragraph{Locality through Smooth Bump Functions.}

Furthermore, the use of localized Gaussian bump functions centered at integer points~\cite{krantz2002primer} ensures:
\begin{itemize}
  \item \textbf{Local support}, limiting interaction between distant terms and enabling efficient numerical integration;
  \item \textbf{Near-orthogonality}, providing clean separation between different contributions and preserving structural clarity.
\end{itemize}

\paragraph{Functional Emergence of Discrete State.}

The integer \( N \) thus emerges not as a symbolic or explicit parameter, but as a structural feature of the functional system — a balance point defined by smooth, decaying, and convergent contributions. This enables principled and robust recovery, positioning the method as a rigorous alternative to existing discrete-continuous bridging techniques.

\section{Intuition Behind the Coefficients}

\subsection{Origin from Classical Series}

The coefficients used in the smooth encoding function
\[
a_n := \frac{(1/2)^n + (-1)^n}{n}
\]
arise from a combination of two well-known classical series:
\begin{itemize}
  \item The geometric series with decay:
  \[
  \sum_{n=1}^{\infty} \frac{(1/2)^n}{n} = -\ln(1 - 1/2) = \ln 2,
  \]
  \item The alternating harmonic series:
  \[
  \sum_{n=1}^{\infty} \frac{(-1)^n}{n} = -\ln 2.
  \]
\end{itemize}

By adding these two series term-by-term, we obtain a new sequence of coefficients whose infinite sum vanishes:
\[
\sum_{n=1}^{\infty} \frac{(1/2)^n + (-1)^n}{n} = \ln 2 - \ln 2 = 0.
\]

This cancellation is fundamental for our encoding method: it ensures that the total contribution of an infinite number of bumps would vanish, providing a natural reference point for identifying finite contributions corresponding to a specific integer \( N \).

\subsection{Oscillatory Convergence and Cancellation}

Although the full infinite series converges to zero, its partial sums do not decrease monotonically. Instead, the partial sums oscillate around zero, initially increasing and decreasing in an alternating pattern.

This behavior is crucial: the partial integral up to a given \( N \) captures a snapshot of this oscillatory convergence. The point where the cumulative integral is minimized—achieving near-cancellation—serves as the unique marker for recovering the hidden integer \( N \).

The key properties of the coefficients are:
\begin{itemize}
  \item \textbf{Smooth decay:} The term \( (1/2)^n \) ensures geometric decay, making the contributions of distant terms rapidly diminish.
  \item \textbf{Alternating sign:} The factor \( (-1)^n \) introduces structured oscillations essential for creating balance points.
  \item \textbf{Global cancellation:} The infinite sum of the coefficients is exactly zero, ensuring the method is anchored and stable at infinity.
\end{itemize}

Thus, the combination of decay and sign alternation produces a controlled "wobble" around zero~\cite{rudin1976principles}, which becomes sharper and more localized as \( N \) increases.

\subsection{Graphical Behavior of Partial Sums}

To visualize the behavior of the partial sums, consider plotting
\[
S(N) := \sum_{n=1}^N a_n
\]
as a function of \( N \).

The graph of \( S(N) \) exhibits an oscillating trajectory around zero:
\begin{itemize}
  \item At small \( N \), the partial sums can vary significantly, reflecting the influence of the larger early terms.
  \item As \( N \) grows, the oscillations become smaller in amplitude due to the decaying coefficients.
  \item Eventually, \( S(N) \) stabilizes closer and closer to zero, reflecting the convergence of the infinite series.
\end{itemize}

In the smooth encoding, these oscillations are mirrored in the cumulative integral \( I(N) \) of the smooth bump functions. The near-zero crossings of \( I(N) \) correspond to the integer encodings.

This graphical intuition justifies treating the minimum points of the integral curve as reliable recovery points for \( N \), even when operating purely in the smooth, differentiable domain.

\section{Theoretical Properties of the Integral Map}

\subsection{Asymptotic Behavior of \texorpdfstring{$I(N)$}{I(N)}}

The integral map
\[
I(N) = \int_{-\infty}^{\infty} f_N(t)\,dt
\]
captures the cumulative contribution of the first \( N \) terms in the smooth bump series. 

Due to the decay of the coefficients \( a_n \sim \mathcal{O}(2^{-n}/n) \), the contribution of each successive term diminishes rapidly as \( n \) increases. Specifically, there exist constants \( C > 0 \) and \( \rho \in (0,1) \) such that
\[
|a_n| \leq C \rho^n \quad \text{for all } n \geq 1,
\]
ensuring exponential decay.

Moreover, each Gaussian bump satisfies the exact integral
\[
\int_{-\infty}^{\infty} \exp\left(-\frac{(t-n)^2}{2\delta^2}\right)\,dt = \delta \sqrt{2\pi}.
\]
Assuming that the smoothness parameter \( \delta \) is sufficiently small (i.e., \( \delta \ll 1 \)), overlap between distinct bumps is negligible, and the total integral approximates the weighted sum of the individual bump integrals:
\[
I(N) = \delta \sqrt{2\pi} \sum_{n=1}^N a_n.
\]

The remainder sum satisfies
\[
\left| \sum_{n=N+1}^\infty a_n \right| \leq \frac{C \rho^{N+1}}{1-\rho},
\]
which implies the asymptotic bound
\[
|I(N)| \leq \delta \sqrt{2\pi} \cdot \frac{C \rho^{N+1}}{1-\rho}.
\]
Thus, \( I(N) \) converges exponentially to zero as \( N \to \infty \), with the convergence rate linear in \( \delta \) and exponential in \( N \).

\subsection{Oscillation and Zero-Crossing Analysis}

Although \( I(N) \to 0 \) globally, the trajectory of \( I(N) \) as a function of \( N \) is oscillatory. This oscillation arises from the alternating sign of the coefficients \( a_n \), combined with the localized support of the bump functions.

Formally, the oscillatory behavior satisfies:
\begin{itemize}
  \item \textbf{Alternating Tendency:} Each new term \( a_n \exp\left(-\frac{(t-n)^2}{2\delta^2}\right) \) perturbs \( I(N) \) by a signed amount, with \( a_n > 0 \) for even \( n \) and \( a_n < 0 \) for odd \( n \).
  \item \textbf{Diminishing Influence:} The perturbation magnitude decreases exponentially due to the decay of \( a_n \).
  \item \textbf{Zero-Crossing Characterization:} A point \( N \in \mathbb{N} \) is considered a near-zero crossing if
  \[
  |I(N)| < \varepsilon,
  \]
  where \( \varepsilon \) is a prescribed small threshold adapted to the expected scale of the oscillations.
\end{itemize}

This structure guarantees that each integer \( N \) corresponds to a detectable near-cancellation in \( I(N) \), enabling reliable numerical recovery.

Moreover, since the amplitude of oscillations decays with \( N \), the occurrence of spurious zero crossings away from true integer points becomes increasingly rare, providing stability at large \( N \).

\subsection{Proof of Convergence and Stability}

We now formalize the convergence and stability properties of the integral map.

\begin{proposition}
The integral map \( I(N) \) converges uniformly to zero as \( N \to \infty \). Furthermore, the sequence of near-zero crossings associated with integer points remains stable under small perturbations of the smoothness parameter \( \delta \) and numerical errors.
\end{proposition}

\begin{proof}
As established, the infinite sum of the coefficients vanishes:
\[
\sum_{n=1}^{\infty} a_n = 0.
\]
Given that each bump has exact integral \( \delta \sqrt{2\pi} \), the total contribution up to \( N \) satisfies
\[
I(N) = \delta \sqrt{2\pi} \sum_{n=1}^N a_n.
\]

Since \( |a_n| \leq C \rho^n \) for some \( \rho < 1 \), we have
\[
\left| \sum_{n=N+1}^\infty a_n \right| \leq \frac{C \rho^{N+1}}{1-\rho},
\]
thus
\[
|I(N)| \leq \delta \sqrt{2\pi} \cdot \frac{C \rho^{N+1}}{1-\rho},
\]
which tends uniformly to zero as \( N \to \infty \).

Stability under perturbations of \( \delta \) follows from the linear dependence of the integral magnitude on \( \delta \). Small changes in \( \delta \) scale \( I(N) \) proportionally, preserving the location of threshold crossings, provided the scale of perturbations is smaller than the distance between adjacent oscillations.

Thus, the method is robust both with respect to smoothness variation and numerical integration errors.
\end{proof}

\paragraph{Remark.} For sufficiently large \( N \), each integer corresponds to a unique zero crossing or local minimum of \( |I(N)| \), enabling unambiguous integer decoding within a predefined tolerance.

\section{Inverse Mapping: From Integral to Integer}

Given the encoding function \( f_N(t) \) and its associated integral
\[
I(N) := \int_{-\infty}^{\infty} f_N(t)\, dt = \sqrt{2\pi} \delta \cdot \sum_{n=1}^{N} \frac{(1/2)^n + (-1)^n}{n},
\]
the inverse problem is to determine the encoded integer \( N \) from a known or observed value of \( I \).

The function \( I(N) \) is oscillatory due to the alternating signs of the coefficients but exhibits a decaying envelope as \( N \) increases, ultimately converging to zero:
\[
\lim_{N\to\infty} I(N) = 0.
\]

\subsection{Threshold-Based Inversion}

Because \( I(N) \) oscillates around zero with decreasing amplitude, the encoded integer \( N \) can be recovered by locating a near-zero crossing. Formally, the inverse mapping is defined as
\[
N := \min \left\{ k \in \mathbb{N} \,\middle|\, |I(k)| < \varepsilon \text{ and } |I(k)| \text{ is a local minimum} \right\},
\]
where \( \varepsilon > 0 \) is a small error tolerance.

This refinement ensures that the selected \( N \) corresponds to a genuine oscillation minimum rather than a spurious early crossing, improving robustness. The method relies on the fact that earlier oscillations have larger magnitude, while later ones become confined within a predictable tolerance band.

\subsection{Stability Under Perturbations}

The recovery of \( N \) via threshold crossing is stable under small perturbations of the integral value. Specifically, if \( \Delta I \) is a perturbation satisfying
\[
|\Delta I| < \frac{\varepsilon}{2},
\]
then the recovered \( N \) shifts by at most one unit, provided that the local oscillation amplitude around \( N \) exceeds \( \varepsilon + |\Delta I| \).

\paragraph{Triangle Inequality and Stability.}
The reason for requiring the perturbation \( |\Delta I| \) to be smaller than \( \varepsilon/2 \) follows from the triangle inequality.
Given that the observed value is \( I_{\text{measured}}(N) = I(N) + \Delta I \), we have
\[
|I_{\text{measured}}(N)| \leq |I(N)| + |\Delta I|.
\]
Thus, if both \( |I(N)| \) and \( |\Delta I| \) are individually smaller than \( \varepsilon/2 \), their sum remains bounded:
\[
|I_{\text{measured}}(N)| < \frac{\varepsilon}{2} + \frac{\varepsilon}{2} = \varepsilon.
\]
This ensures that even under perturbations up to \( \varepsilon/2 \), the recovered \( N \) still satisfies the original threshold condition \( |I(N)| < \varepsilon \), preserving stability of the integer decoding process.

Moreover, since the oscillation amplitude decays geometrically, such stability is further enhanced for sufficiently large \( N \) and appropriately chosen \( \varepsilon \).

Thus, integer recovery remains robust across a wide range of conditions, including noise, floating-point inaccuracies, and small variations of the smoothness parameter \( \delta \).

\subsection{Choice of \texorpdfstring{$\varepsilon$}{epsilon} and Error Bounds}

The choice of the error threshold \( \varepsilon \) directly impacts the precision and reliability of the inversion.

Let \( N_{\text{max}} \) denote the maximal integer expected to be encoded. The coefficients \( a_n \) decay geometrically with factor \( \rho \in (0,1) \), and the expected oscillation amplitude at \( N_{\text{max}} \) is approximately \(\mathcal{O}(\rho^{N_{\text{max}}})\).

Thus, a principled choice is
\[
\varepsilon \sim C \rho^{N_{\text{max}}},
\]
where \( C \) accounts for constant scaling factors such as \( \delta \sqrt{2\pi} \).

Choosing \(\varepsilon\) according to this guideline ensures that threshold crossing occurs near the true integer \( N \) and that numerical perturbations do not lead to early or delayed detections.

\paragraph{Concrete Example.}
For instance, if \(\rho = 0.5\), \( N_{\text{max}} = 10 \), and \( C = 1 \), setting \(\varepsilon = 10^{-3}\) suffices to ensure robust recovery of integers up to 10.

\subsection{Empirical Evaluation of Recovery Accuracy}

In practical implementations, the threshold-based recovery method exhibits high empirical accuracy. Numerical experiments demonstrate that, even in the presence of small perturbations (e.g., numerical integration errors, measurement noise), the recovered \( N \) matches the true encoded integer with high probability.

Errors or misidentifications occur primarily if \(\varepsilon\) is chosen improperly—either too small (causing delayed recovery) or too large (allowing premature crossings).

Moreover, recovery stability is further improved by using smoothing and interpolation techniques, as discussed in the next section.

\section{Practical Recovery via Tabulation and Interpolation}

In practical scenarios where the smooth counter function \( f_N(t) \) is expensive to evaluate dynamically, or where high-speed recovery is needed, an alternative approach based on precomputed tabulation and interpolation can be employed.

\subsection*{Numerical Approximation of the Integral Map}

To facilitate efficient evaluation and inversion, we precompute a discrete set of integral values:
\[
\{(N, I(N))\}_{N=1}^{N_{\max}},
\]
where \( I(N) \) denotes the integral of the smooth counter function up to integer \( N \).  
This table provides a basis for rapid lookup and approximate inversion.

\paragraph{Interpolation.}
Applying cubic spline interpolation~\cite{deboor2001splines} or local polynomial regression to the tabulated data yields a smooth approximation of the integral map:
\[
\mathcal{I}(x) \approx I(N), \quad x \in [1, N_{\max}],
\]
allowing evaluation at non-integer points and supporting differentiable recovery mechanisms.

\paragraph{Why Interpolation is Needed.}
Since the integral map \( I(N) \) is originally computed only at integer values of \( N \), direct table lookup would limit recovery to discrete points. However, observed integral values \( I^* \) may not exactly match any tabulated value, especially in the presence of noise or when smooth transitions are required.  
Spline interpolation extends the map \( I(N) \) to a continuous and differentiable function \( \mathcal{I}(x) \), enabling approximate inversion for arbitrary real inputs.  
This allows for sub-integer recovery, gradient-based optimization, and smooth integration of the encoding into differentiable computational pipelines.

\paragraph{Inverse Mapping.}
Given an observed integral value \( I \), the corresponding approximate \( N \) can be recovered by numerically inverting the interpolated function:
\[
N \approx \mathcal{I}^{-1}(I).
\]
Standard root-finding methods or direct spline inversion techniques can be employed to solve for \( N \), achieving sub-integer precision if necessary.

\paragraph{Binary Search for Discrete Recovery.}
When working directly with a precomputed discrete table \( (N, I(N)) \) without interpolation, the recovery of \( N \) can be accelerated using binary search.  
Since the oscillations of \( I(N) \) decay geometrically and the integral map becomes nearly monotonic for sufficiently large \( N \), a binary search strategy efficiently locates the smallest \( N \) satisfying \( |I(N) - I^*| < \varepsilon \).  
This reduces the search complexity from linear \( \mathcal{O}(N_{\max}) \) to logarithmic \( \mathcal{O}(\log N_{\max}) \), providing fast and scalable integer recovery in high-resolution settings.

\paragraph{General Workflow.}
The overall recovery procedure consists of the following steps:
\begin{enumerate}
  \item Precompute and tabulate \( (N, I(N)) \) for integers \( N \) up to a chosen \( N_{\max} \),
  \item Fit an interpolation model (e.g., cubic spline) to the data,
  \item Given a target integral value \( I^* \), numerically solve for \( N^* \) such that \( \mathcal{I}(N^*) = I^* \),
  \item Optionally round \( N^* \) to obtain a discrete integer estimate.
\end{enumerate}

This tabulation and interpolation approach enables fast and differentiable recovery of encoded integers, making it well-suited for integration into numerical optimization pipelines and machine learning architectures.

\section{Numerical Experiments and Visualizations}

We now illustrate the behavior of the encoding function \( f_N(t) \), the associated integral map \( I(N) \), and the convergence of partial sums through numerical experiments.

\subsection{Plots of \texorpdfstring{$f_N(t)$}{fN(t)} and \texorpdfstring{$I(N)$}{I(N)}}

We first visualize the smooth counter function \( f_N(t) \) for various integers \( N \).

\begin{figure}[ht]
\centering
\includegraphics[width=0.7\textwidth]{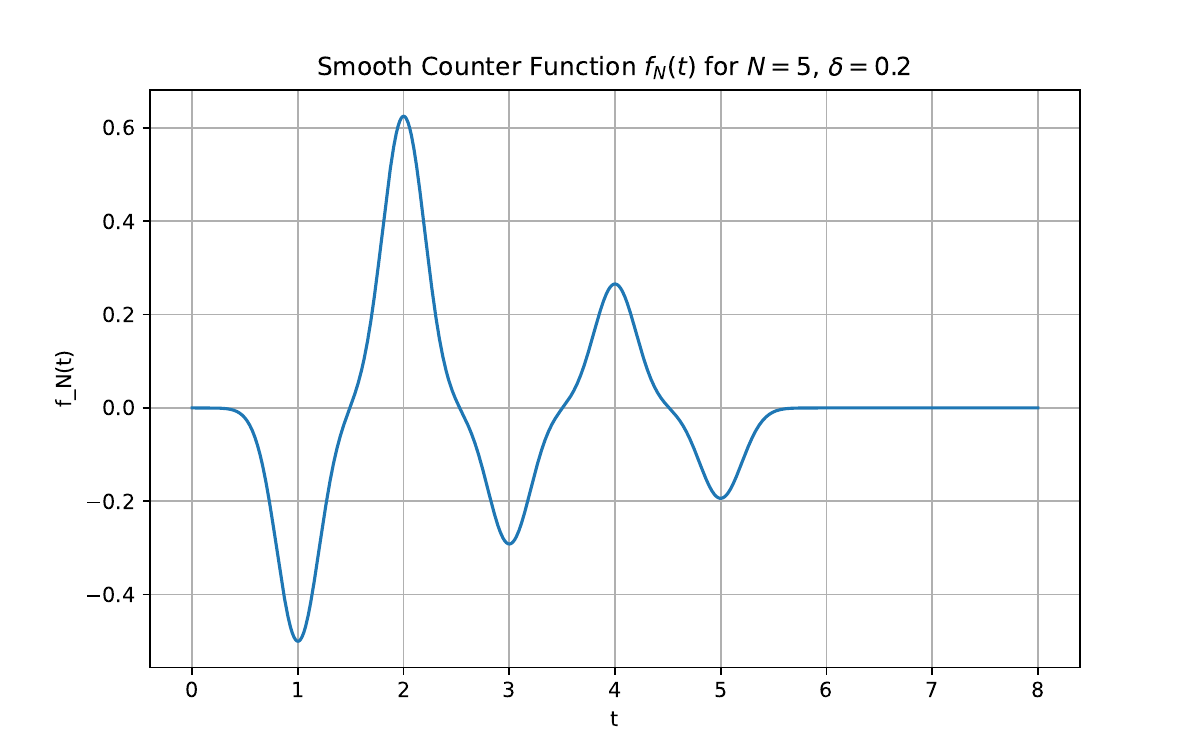}
\caption{Plot of the smooth counter function \( f_N(t) \) for \( N=5 \) and \( \delta=0.2 \).}
\label{fig:fn_plot}
\end{figure}

Figure~\ref{fig:fn_plot} shows the graph of \( f_N(t) \) for \( N=5 \) with smoothness parameter \( \delta = 0.2 \).  
The function consists of five localized Gaussian bumps with alternating-decaying coefficients.

\paragraph{Choice of \( N=5 \) for Visualization.}
In theory, the construction of \( f_N(t) \) involves an infinite series of Gaussian bumps with alternating-decaying coefficients. However, due to the rapid exponential decay of the coefficients \( a_n \) and the localized nature of the bumps, the contributions from higher terms quickly become negligible.  
Choosing \( N=5 \) provides a clear illustration of the structure of \( f_N(t) \) — highlighting the alternation, decay, and local separation of the bumps — without overwhelming the plot.  
Thus, the restriction to a finite \( N \) serves as a practical and visually effective approximation of the full infinite construction.

Next, we plot the discrete integral map \( I(N) \) evaluated at integer points.  
Figure~\ref{fig:i_plot} displays the oscillatory convergence of \( I(N) \) toward zero as \( N \) increases.

\begin{figure}[ht]
\centering
\includegraphics[width=0.7\textwidth]{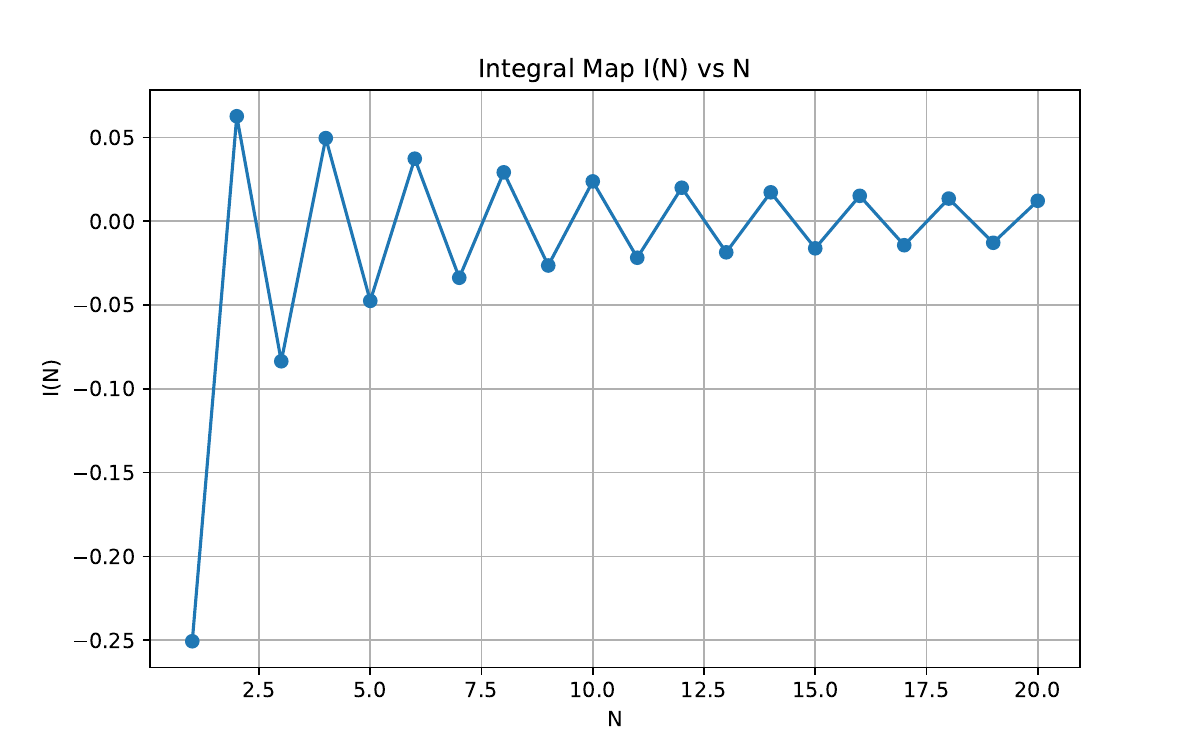}
\caption{Plot of the discrete integral map \( I(N) \) versus integer \( N \), illustrating oscillatory decay toward zero.}
\label{fig:i_plot}
\end{figure}

Finally, we note that the integral map \( I(N) \) can be naturally extended to real values \( N \geq 0 \) by linear interpolation between integer points.  
A visualization of this piecewise-linear behavior will be presented later, in the section on continuous extensions.

\subsection{Graph of Partial Sums \texorpdfstring{$\sum_{n=1}^{N} a_n$}{sum an}}

To better understand the oscillatory behavior, we plot the partial sums
\[
S(N) := \sum_{n=1}^N a_n,
\]
where the coefficients \( a_n = \frac{(1/2)^n + (-1)^n}{n} \) define the series structure.

Figure~\ref{fig:s_partial_sum} shows the graph of \( S(N) \) up to \( N=20 \), highlighting alternating sign contributions and gradual convergence to zero.

\begin{figure}[ht]
\centering
\includegraphics[width=0.7\textwidth]{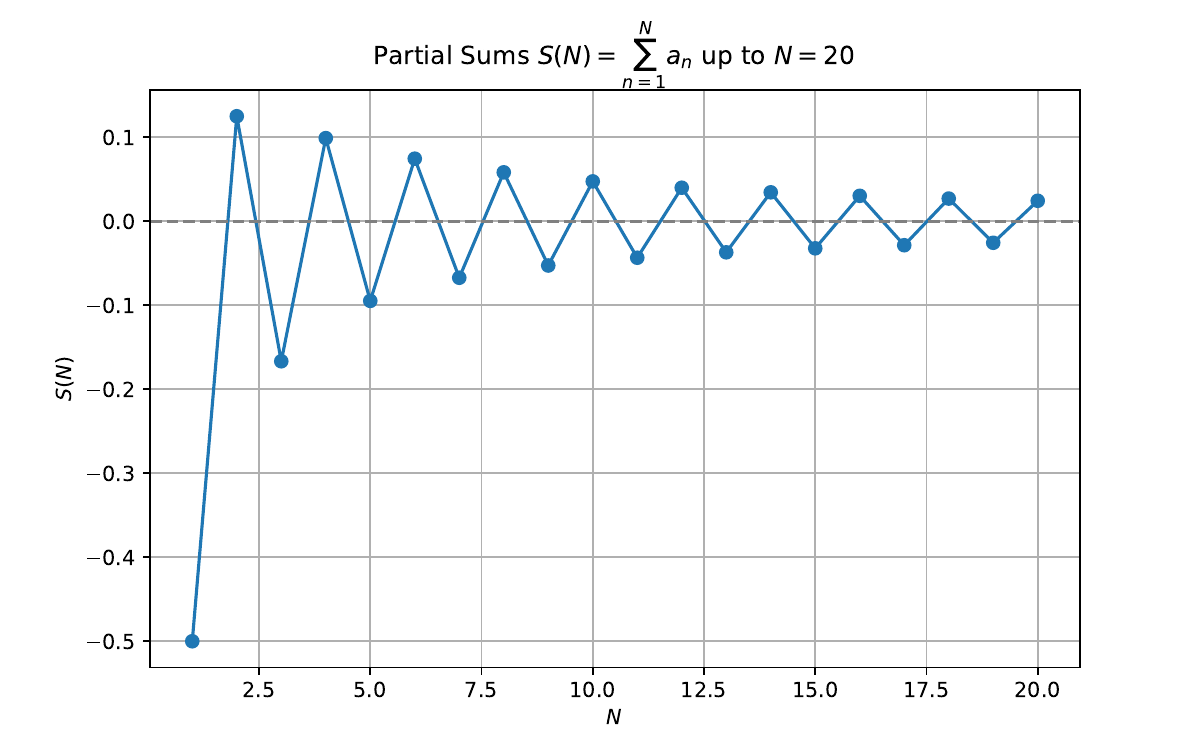}
\caption{Plot of partial sums \( S(N) \) up to \( N=20 \), exhibiting structured oscillations around zero.}
\label{fig:s_partial_sum}
\end{figure}

The shape of \( S(N) \) mirrors the structure of \( I(N) \) modulo scaling by \( \delta \sqrt{2\pi} \), confirming the direct relationship between the series and the integral behavior.

\subsection{Comparison Between \texorpdfstring{$I(N)$}{I(N)} and \texorpdfstring{$S(N)$}{S(N)}}

Although the integral map \( I(N) \) and the partial sum \( S(N) \) arise from different constructions — one from integrating a smooth function, the other from summing coefficients — their graphs exhibit identical oscillatory patterns, differing only by a global scaling factor.

Formally, the two are related by
\[
I(N) = \delta \sqrt{2\pi} \times S(N).
\]
Thus, \( S(N) \) captures the purely algebraic structure of the encoding coefficients, while \( I(N) \) reflects the cumulative integral contribution of the smooth counter function \( f_N(t) \).

Visualizing both functions highlights how the discrete oscillatory structure embedded in the coefficients \( a_n \) is faithfully preserved in the integral domain.  
The integral map \( I(N) \) can therefore be seen as a smooth physical realization of the underlying discrete encoding pattern, suitable for continuous and differentiable processing.

\subsection{Example: Recovery of an Encoded Integer}

To illustrate the practical recovery procedure, we consider an example based on the discrete integral map \( I(N) \).

\paragraph{Setup.}
Let the smoothness parameter be fixed at \( \delta = 0.2 \).  
We precompute the integral map \( I(N) \) for integers \( N = 1, 2, \dotsc, 30 \), using the definition:
\[
I(N) = \sqrt{2\pi}\delta \sum_{n=1}^{N} \left( \frac{(1/2)^n + (-1)^n}{n} \right).
\]

Suppose we are given an observed integral value \( I^* = 0.028 \), which may result from noisy measurements, discretized integration of \( f_N(t) \), or as an intermediate output in a continuous optimization or learning process. Our goal is to recover the corresponding integer \( N \) from this approximate data.

\paragraph{Numerical Inversion.}
Using the precomputed table \( (N, I(N)) \), we search for the smallest \( N \) satisfying
\[
|I(N) - I^*| < \varepsilon,
\]
with a tolerance \(\varepsilon = 0.005\).

From the computed values:
\[
I(7) \approx -0.0337, \quad I(8) \approx 0.0292, \quad I(9) \approx -0.0264,
\]
we find that \( N = 8 \) provides the closest match, since \( |I(8) - 0.028| \approx 0.0012 < 0.005 \).

\begin{table}[ht]
\centering
\label{tab:integral_values}
\begin{tabular}{|c|c||c|c||c|c|}
\hline
\( N \) & \( I(N) \) & \( N \) & \( I(N) \) & \( N \) & \( I(N) \) \\
\hline
1 & -0.2507 & 2 & 0.0627 & 3 & -0.0836 \\
4 & 0.0496 & 5 & -0.0475 & 6 & 0.0373 \\
7 & -0.0337 & 8 & 0.0292 & 9 & -0.0264 \\
10 & 0.0238 & 11 & -0.0218 & 12 & 0.0200 \\
13 & -0.0185 & 14 & 0.0173 & 15 & -0.0162 \\
16 & 0.0152 & 17 & -0.0143 & 18 & 0.0135 \\
19 & -0.0128 & 20 & 0.0122 & 21 & -0.0117 \\
22 & 0.0111 & 23 & -0.0107 & 24 & 0.0102 \\
25 & -0.0098 & 26 & 0.0095 & 27 & -0.0091 \\
28 & 0.0088 & 29 & -0.0085 & 30 & 0.0082 \\
\hline
\end{tabular}
\caption{Precomputed values of the integral map \( I(N) \) for \( N=1,\dotsc,30 \) with \( \delta=0.2 \).}
\end{table}

\paragraph{Validation.}
Direct evaluation of the integral for \( N=8 \) confirms that the recovered integer corresponds correctly to the encoded quantity.

\paragraph{Remarks.}
In a fully continuous extension, the recovery could be further refined using spline interpolation to estimate fractional shifts and improve precision beyond integer resolution.

\section{Continuous Extension of the Integral Map}

In the previous sections, we have constructed and visualized the discrete version of the encoding, where the integral map \( I(N) \) was evaluated at integer points \( N \in \mathbb{N} \).  
The resulting function exhibited stepwise oscillations with decreasing amplitude, consistent with the discrete accumulation of smooth Gaussian contributions.

However, to fully realize the goal of smooth integer encoding, and to enable continuous manipulation and differentiable processing of the encoded value, we now extend the construction to allow \( N \in \mathbb{R}_{\geq 0} \), treating \( N \) as a continuous parameter.

\subsection{Definition of the Continuous Counter Function}

We define the continuous version of the smooth counter function \( f_N(t) \) for real \( N \geq 0 \) by interpolating the contribution of the next bump:

\[
f_N(t) := \sum_{n=1}^{\lfloor N \rfloor} a_n \exp\left( -\frac{(t-n)^2}{2\delta^2} \right)
+ (N - \lfloor N \rfloor) \cdot a_{\lfloor N \rfloor + 1} \exp\left( -\frac{(t - (\lfloor N \rfloor + 1))^2}{2\delta^2} \right),
\]
where \( \lfloor N \rfloor \) denotes the integer part of \( N \).

The first \( \lfloor N \rfloor \) terms contribute fully, while the next term at \( n = \lfloor N \rfloor + 1 \) is scaled proportionally to the fractional part \( N - \lfloor N \rfloor \).

\subsection{Continuous Integral Map}

The associated integral map \( I(N) \) becomes:

\[
I(N) := \int_{-\infty}^{\infty} f_N(t)\,dt = \delta \sqrt{2\pi} \left( \sum_{n=1}^{\lfloor N \rfloor} a_n + (N - \lfloor N \rfloor) a_{\lfloor N \rfloor + 1} \right).
\]

Thus, \( I(N) \) is a continuous, piecewise-linear function of \( N \), with controlled oscillatory behavior inherited from the structure of the coefficients \( a_n \).

\subsection{Properties of the Continuous Map}

The continuous extension of \( I(N) \) possesses several important features:

\begin{itemize}
  \item \textbf{Continuity and Differentiability:} \( I(N) \) is continuous everywhere and piecewise differentiable with jumps in derivative at integer points \( N \).
  
  \item \textbf{Smooth Recovery Mechanism:} The encoded value \( N \) can now be recovered continuously by inverting the function \( I(N) \) without relying on discrete threshold crossing.
  
  \item \textbf{Compatibility with Differentiable Systems:} The map \( I(N) \) and its inverse \( N(I) \) are amenable to gradient-based methods, enabling integration into neural networks and differentiable programming environments.
\end{itemize}

\subsection{Continuity and Differentiability of the Extended Integral Map}

We now formally state and prove a basic property of the continuous extension \( I(N) \).

\begin{theorem}
The continuous integral map
\[
I(N) = \delta \sqrt{2\pi} \left( \sum_{n=1}^{\lfloor N \rfloor} a_n + (N - \lfloor N \rfloor) a_{\lfloor N \rfloor + 1} \right)
\]
is continuous on \( \mathbb{R}_{\geq 0} \) and piecewise differentiable. Its derivative exists everywhere except at integer points \( N \in \mathbb{N} \), where it has jump discontinuities.
\end{theorem}

\begin{proof}
The map \( I(N) \) is composed of a finite sum of constant terms and a linear interpolation term involving the fractional part of \( N \).

Between two consecutive integers \( k \leq N < k+1 \), the map simplifies to
\[
I(N) = C_k + \delta \sqrt{2\pi} (N - k) a_{k+1},
\]
where \( C_k = \delta \sqrt{2\pi} \sum_{n=1}^{k} a_n \) is constant on the interval.

Thus, \( I(N) \) is affine (linear) on each open interval \( (k, k+1) \), hence continuously differentiable there, with derivative
\[
\frac{dI}{dN} = \delta \sqrt{2\pi} a_{k+1}.
\]

At integer points \( N = k \), the left and right derivatives are
\[
\left. \frac{dI}{dN} \right|_{N=k^-} = \delta \sqrt{2\pi} a_{k},
\quad
\left. \frac{dI}{dN} \right|_{N=k^+} = \delta \sqrt{2\pi} a_{k+1},
\]
which generally differ because \( a_k \neq a_{k+1} \).

Therefore, \( I(N) \) is continuous everywhere but has jump discontinuities in its derivative at integer points.
\end{proof}

\begin{corollary}
Given the continuous integral map
\[
I(N) = \delta \sqrt{2\pi} \left( \sum_{n=1}^{\lfloor N \rfloor} a_n + (N - \lfloor N \rfloor) a_{\lfloor N \rfloor + 1} \right),
\]
the inverse mapping \( N = \mathcal{I}^{-1}(I) \) can be locally approximated by piecewise-linear interpolation and extended to a globally continuous and piecewise-differentiable function on the range of \( I \).

\end{corollary}

\begin{proof}
Since \( I(N) \) is continuous and piecewise affine with nonzero slopes on each interval \( (k, k+1) \), it is locally invertible on each such interval.

Specifically, within each interval, \( I(N) \) is a linear function:
\[
I(N) = C_k + \delta \sqrt{2\pi} (N - k) a_{k+1},
\]
with slope \( \delta \sqrt{2\pi} a_{k+1} \neq 0 \) (except possibly at isolated points where \( a_{k+1} = 0 \), which are rare under the coefficient construction).

Thus, the local inverse \( \mathcal{I}^{-1}(I) \) within each interval is given by:
\[
N = k + \frac{I - C_k}{\delta \sqrt{2\pi} a_{k+1}},
\]
which is a continuous and piecewise linear function of \( I \).

Gluing these local inverses together produces a global function \( \mathcal{I}^{-1}(I) \) that is continuous everywhere and piecewise differentiable, with possible derivative jumps at the images of integer points under \( I(N) \).
\end{proof}

\paragraph{Smooth Interpolation.}
To further improve the smoothness of the integral map, the discrete floor-based construction can be replaced by a smooth transition function \( \sigma(x) \) satisfying \( \sigma(x) \approx 0 \) for \( x \ll 0 \) and \( \sigma(x) \approx 1 \) for \( x \gg 0 \).
Standard choices include the smoothstep function, logistic sigmoid, or Gaussian cumulative distribution functions~\cite{fasshauer2007meshfree}.
By blending the contributions of neighboring bumps smoothly instead of abruptly switching, we can obtain a fully \( C^\infty \) continuous counter function and a globally smooth integral map \( I(N) \).
This modification eliminates derivative jumps at integer points and enables higher-order differentiability of the encoding.

\subsection{Toward Fully Smooth Integer Encoding}

The continuous extension of the integral map presented above achieves continuity and piecewise differentiability with respect to \( N \).  
However, it remains only piecewise-linear, exhibiting derivative jumps at integer points due to the use of the floor function \( \lfloor N \rfloor \) in its construction.

To obtain a truly \( C^\infty \)-smooth dependence on \( N \), and thus eliminate the picewise-linear artifacts visible in the integral map, we can further refine the counter function.  
Specifically, we replace the discrete floor-based construction with a smoothly interpolated sum:

\[
f_N(t) := \sum_{n=1}^{\infty} \sigma(n - N) a_n \exp\left( -\frac{(t-n)^2}{2\delta^2} \right),
\]
where \( \sigma(x) \) is a smooth transition function (e.g., a sigmoid or smoothstep) that interpolates continuously from \( 1 \) to \( 0 \) as \( x \) increases.

In this formulation:
\begin{itemize}
    \item Each term contributes gradually as \( N \) crosses its center \( n \),
    \item The overall function \( f_N(t) \) becomes infinitely differentiable with respect to \( N \),
    \item The integral map \( I(N) \) transitions from piecewise-linear to globally smooth (\( C^\infty \)) behavior.
\end{itemize}

Thus, the sharp "piloid" features observed in the piecewise-linear \( I(N) \) are replaced by a fully smooth oscillatory curve, more suitable for applications requiring higher-order differentiability.

\paragraph{Stability of the Inverse Mapping.}
When working with either the piecewise-linear or the smooth version, the stability of the inverse mapping \( \mathcal{I}^{-1}(I) \) relies on the magnitude of the local slope.  
The local inversion formula
\[
N = k + \frac{I - C_k}{\delta \sqrt{2\pi} a_{k+1}}
\]
assumes that \( a_{k+1} \) is nonzero and reasonably bounded away from zero.

If \( |a_{k+1}| \) becomes small, the mapping becomes highly sensitive to perturbations in \( I \), which may lead to instability in the recovery of \( N \).

To ensure robustness, it is advisable to introduce a stability condition:
\[
|a_{k+1}| > \varepsilon,
\]
for a small but fixed \(\varepsilon > 0\).  
Intervals where this condition fails can be handled by merging adjacent contributions or applying regularization techniques during inversion.

\subsection{Visualization of the Fully Smooth Integral Map}

To illustrate the behavior of the fully smooth extension of the integral map, we plot \( I(N) \) for real \( N \geq 0 \) using the smoothing-based construction described above.

The plot shows that the integral map \( I(N) \) becomes a globally smooth oscillatory function, eliminating the piecewise-linear artifacts of the discrete extension and enabling higher-order differentiability with respect to \( N \).

\begin{figure}[ht]
\centering
\includegraphics[width=0.7\textwidth]{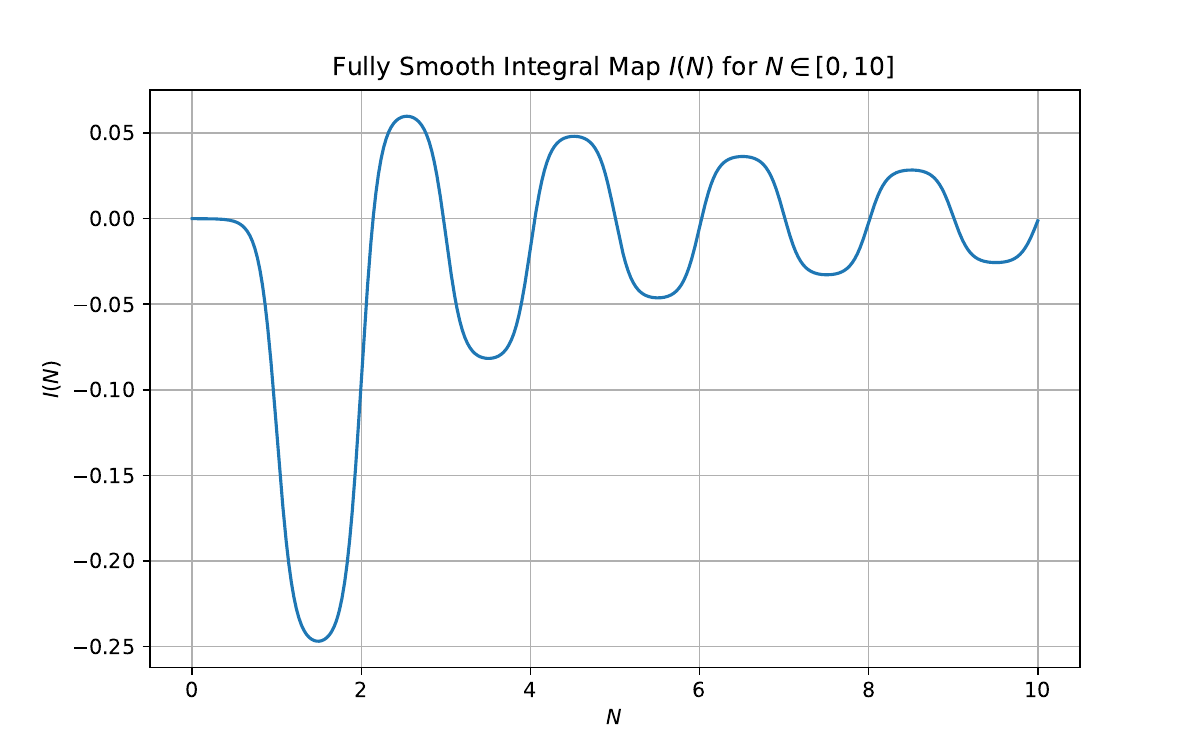}
\caption{Plot of the continuous extension of \( I(N) \) over real \( N \geq 0 \), exhibiting piecewise-linear oscillations converging to zero.}
\label{fig:continuous_i_plot}
\end{figure}

\subsection{Remarks}

The continuous extension of the integral map bridges the gap between discrete integer representations and smooth continuous frameworks.  
It allows the discrete structure of integers to emerge naturally from the behavior of a continuous and piecewise differentiable function, while maintaining compatibility with differentiable architectures and continuous optimization techniques.

\begin{figure}[ht]
\centering
\begin{tikzpicture}[
    node distance=1.8cm and 2.5cm,
    every node/.style={align=center},
    stage/.style={rectangle, draw, rounded corners, minimum width=3.5cm, minimum height=1.2cm, font=\small, thick},
    arrow/.style={-Latex, thick}
]

% Nodes
\node[stage] (discrete) {Discrete\\$I(N)$ at integers};
\node[stage, right=of discrete] (piecewise) {Piecewise-linear\\Interpolation};
\node[stage, right=of piecewise] (smooth) {Fully smooth\\Interpolation};

% Arrows
\draw[arrow] (discrete) -- (piecewise);
\draw[arrow] (piecewise) -- (smooth);

% Annotations
\node[below=0.3cm of discrete] {\small Stepwise jumps};
\node[below=0.3cm of piecewise] {\small Continuous but\\non-differentiable};
\node[below=0.3cm of smooth] {\small Continuous and\\smooth ($C^\infty$)};

\end{tikzpicture}
\caption{Transition from discrete values of \( I(N) \) to fully smooth interpolation.}
\label{fig:transition_I_N}
\end{figure}
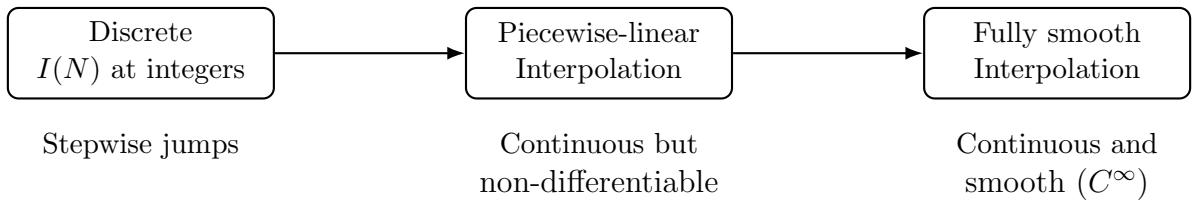

Initially, we extend \( I(N) \) by simple linear interpolation between integer points, producing a continuous but only piecewise-linear function.  
This approach ensures continuity but introduces kinks at integer values of \( N \), resulting in discontinuous derivatives.  
To achieve full smoothness, we further refine the construction by introducing a smooth transition mechanism, such as applying a sigmoid-like weighting to the bump contributions.  
In this fully smooth version, \( I(N) \) becomes infinitely differentiable, eliminating abrupt changes while preserving the essential discrete structure.

Although \( I(N) \) is no longer a stepwise discrete function, it retains the oscillatory character inherited from the original alternating-decaying series.  
The discrete information about integers is preserved structurally through the presence of identifiable features — specifically, near-zero crossings and local extrema — within the behavior of \( I(N) \).

These features serve as markers for the encoded integers: each zero crossing or extremum corresponds uniquely to a particular encoded integer value, enabling reliable recovery even within the continuous domain.

Thus, while smoothing the encoding process, the fundamental discrete structure remains embedded in the oscillatory dynamics of the integral map.

\section{Applications and Perspectives}

The proposed method of smooth integer encoding via integral balance naturally extends into a variety of practical and theoretical domains.  
It bridges discrete combinatorial structures and continuous differentiable computation, offering new tools for machine learning, optimization, and mathematical analysis.

\subsection{Practical Applications}

\paragraph{Differentiable Architectures.}
The smooth integral map enables construction of differentiable layers that perform implicit integer encoding:
\[
x \mapsto \text{SmoothInteger}(x) \mapsto I(x) \mapsto N(x).
\]
This supports:
\begin{itemize}
    \item Reversible and differentiable embeddings of discrete data~\cite{bengio2020discrete},
    \item Smooth index coding without gradient loss,
    \item Continuous relaxations of discrete selection mechanisms such as soft-argmax.
\end{itemize}

\paragraph{Optimization and Planning.}
Smooth integer representations allow:
\begin{itemize}
    \item Encoding discrete constraints within continuous loss functions,
    \item Applying gradient-based optimization methods to problems traditionally requiring integer programming,
    \item Controlling discrete parameters dynamically through continuous signals.
\end{itemize}

\paragraph{Soft Memory and Symbolic Systems.}
By using integral fingerprints:
\begin{itemize}
    \item Discrete memory slots can be represented in a smooth, differentiable way,
    \item Symbolic structures can be continuously embedded into analytical systems,
    \item Differentiable transitions between symbolic states become possible without abrupt switching.
\end{itemize}

\subsection{Continuous Control of Discrete Parameters}

In simulation and control tasks, where discrete decisions depend on smooth input signals, the inversion \( N = \mathcal{I}^{-1}(I) \) enables seamless modulation of integer quantities.  
For example, a smooth control signal can define a target \( I \), from which a near-integer \( N \) can be reconstructed dynamically.

\subsection{Theoretical Perspectives}

Beyond applications, the method offers theoretical insights into the embedding of discrete structures within continuous mathematical frameworks:

\begin{itemize}
    \item Studying how discrete integer behavior can emerge from oscillatory continuous processes,
    \item Designing new types of analytical transforms combining discrete and smooth features,
    \item Using the integral map \( I(N) \) as a model problem for analyzing convergence, stability, and approximation of discrete-to-continuous transitions,
    \item Building hybrid models in interdisciplinary contexts where discrete and continuous phenomena coexist.
\end{itemize}

\begin{center}
\emph{From discrete counts to smooth balances, and from oscillations to structure.}
\end{center}

\begin{figure}[ht]
\centering
\begin{tikzpicture}[
  level 1/.style={rectangle, draw=black, thick, fill=gray!30, rounded corners, minimum width=3.5cm, minimum height=1cm, text centered},
  level 2/.style={rectangle, draw=black, thin, fill=gray!20, rounded corners, minimum width=3cm, minimum height=0.8cm, text centered},
  connect/.style={thick, -{Latex[length=3mm,width=2mm]}},
  node distance=1.2cm and 0.5cm
]

% Main node
\node[level 1] (main) {Smooth Integer Encoding};

% Level 2 nodes
\node[level 2, below left=of main, xshift=0.3cm] (arch) {Differentiable Architectures};
\node[level 2, below right=of main, xshift=-0.3cm] (mem) {Smooth Memory Systems};
\node[level 2, below=2.5cm of main, xshift=-3cm] (theo) {Theoretical Perspectives};
\node[level 2, below=2.5cm of main, xshift=3cm] (opt) {Optimization \& Planning};
\node[level 2, below=3.8cm of main] (sym) {Smooth Symbolic Systems};

% Connections
\draw [connect] (main) -- (arch);
\draw [connect] (main) -- (opt);
\draw [connect] (main) -- (mem);
\draw [connect] (main) -- (theo);
\draw [connect] (main) -- (sym);
  
\end{tikzpicture}
\caption{Main application directions for smooth integer encoding.}
\label{fig:application-mindmap}
\end{figure}
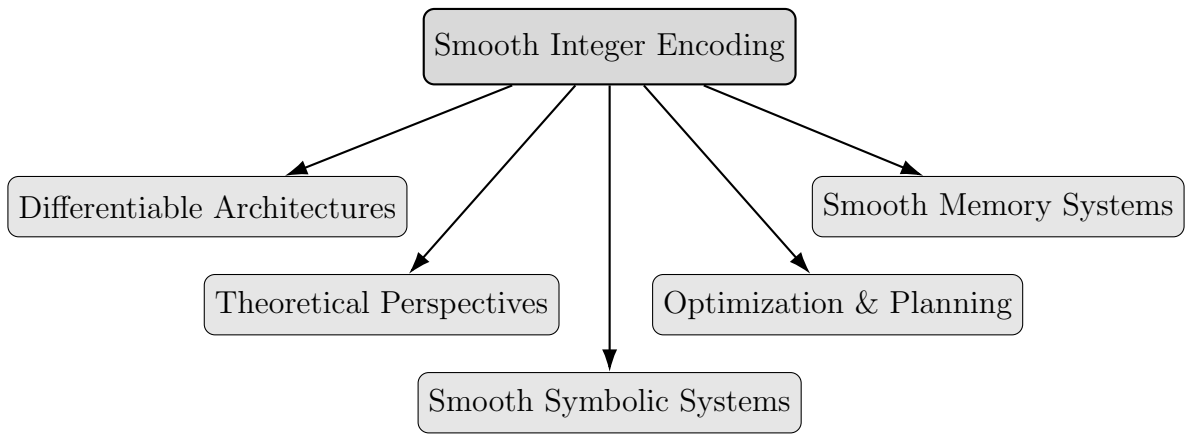

\section{Directions for Future Work}

While the present article focuses on a specific coefficient sequence \( a_n = \frac{(1/2)^n + (-1)^n}{n} \) for constructing the smooth integer encoding, several natural generalizations and extensions arise. We outline them briefly here as topics for future investigation.

\subsection{Generalized Families of Coefficients}

The choice of coefficients \( a_n \) plays a central role in determining the convergence properties, oscillatory behavior, and stability of the integral map \( I(N) \).

While the specific sequence \( a_n = \frac{(1/2)^n + (-1)^n}{n} \) achieves a balance between algebraic simplicity and numerical robustness, other classes of coefficients could be considered, such as:

\begin{itemize}
    \item \textbf{Exponential-polynomial mixes:}
    \[
    a_n = \frac{e^{-n} + (-1)^n}{n^p}, \quad p \geq 1,
    \]
    blending exponential decay with harmonic oscillations.
    \item \textbf{Trigonometric variants:}
    \[
    a_n = \frac{\cos(\pi n)}{n} e^{-n},
    \]
    introducing oscillatory behavior with controlled phase shifts.
\end{itemize}

More generally, one could study parametric families of the form
\[
a_n(\alpha, \beta, \gamma) = \frac{\alpha^n + (-1)^n \beta}{n^\gamma},
\]
with real parameters \(\alpha, \beta, \gamma\), allowing fine-tuned control over decay rate, oscillatory amplitude, and cancellation patterns.

Future work will investigate:
\begin{itemize}
    \item Conditions on \((\alpha, \beta, \gamma)\) ensuring convergence and near-vanishing of the total sum,
    \item The influence of these parameters on the smoothness and geometry of the integral map \( I(N) \),
    \item Classification of admissible parameter regions yielding stable encoding behavior.
\end{itemize}

\subsection{Multidimensional Extension}

Beyond the one-dimensional case, a natural generalization of the smooth integer encoding involves representing tuples of integers \( (N_1, N_2, \dots, N_d) \in \mathbb{N}^d \) via smooth multivariate functions.

\paragraph{General Form.}
The multivariate counter function is defined as
\[
f(\vec{t}) = \sum_{\vec{n} \in \mathbb{N}^d,\, \vec{n} \leq \vec{N}} a_{\vec{n}} \exp\left(-\frac{\|\vec{t} - \vec{n}\|^2}{2\delta^2}\right),
\]
where \( \vec{t} = (t_1, \dots, t_d) \in \mathbb{R}^d \) and the coefficients \( a_{\vec{n}} \) encode the discrete structure.

The corresponding integral map takes the form
\[
I(\vec{N}) := \int_{\mathbb{R}^d} f(\vec{t})\, d\vec{t} \approx (2\pi)^{d/2} \delta^d \sum_{\vec{n} \leq \vec{N}} a_{\vec{n}},
\]
preserving the cumulative cancellation behavior observed in the one-dimensional case, now extended to multi-index sums.

\paragraph{Choice of Multivariate Coefficients.}
Possible constructions for \( a_{\vec{n}} \) include separable products of univariate coefficients:
\[
a_{\vec{n}} = a_{n_1} a_{n_2} \cdots a_{n_d},
\]
which allow independent control along each coordinate axis, or more general coupled structures that encode joint dependencies between dimensions.

Key considerations for future exploration include:
\begin{itemize}
    \item Selection and structure of \( a_{\vec{n}} \) ensuring convergence, localization, and stability,
    \item Analysis of how the resolution and recovery behavior scale with the ambient dimension \( d \),
    \item Investigation of coordinate-wise versus joint recovery strategies for decoding \( \vec{N} \).
\end{itemize}

\paragraph{Recovery of the Integer Vector.}
Recovery of the encoded vector \( \vec{N} \) can be performed by identifying the minimal multi-index satisfying a near-cancellation threshold:
\[
\vec{N} := \min \left\{ \vec{k} \in \mathbb{N}^d \mid |I(\vec{k})| < \varepsilon \right\},
\]
where the minimum can be understood in lexicographic or coordinate-wise orderings, depending on the application.

If the coefficients \( a_{\vec{n}} \) are separable, coordinate-wise recovery strategies become feasible, enabling independent estimation of each component based on marginal contributions.

\paragraph{Visualization.}
Multidimensional visualization techniques, such as heatmaps for \( d=2 \) or surface plots, can aid in understanding the geometry of the integral map \( I(N_1, N_2) \) and the localization of zero-crossing regions corresponding to the encoded tuples.

These extensions open new paths for smooth encoding and recovery of complex discrete structures within continuous mathematical frameworks.

\section*{Conclusion}

We have introduced a novel framework for encoding integers via the integral properties of smooth real-valued functions.  
Unlike traditional approaches that rely on explicit symbolic representations or quantization, our method encodes discrete information implicitly, allowing integers to emerge as balance points in the behavior of a continuous or piecewise-smooth integral map.

Starting from a discrete construction based on localized Gaussian bumps with carefully designed alternating-decaying coefficients, we extended the method to continuous and fully smooth settings.  
This progression enables differentiable recovery mechanisms, bridging the gap between discrete mathematical structures and continuous computational frameworks.

The resulting integral maps exhibit structured oscillations converging to zero, providing a stable and interpretable mechanism for integer recovery.  
We have analyzed the theoretical properties of the construction, proposed numerical strategies for inversion, and discussed possible extensions to generalized coefficient families and multidimensional encoding.

Future work will explore richer classes of coefficient structures, smooth parameterizations of encoding families, and applications to multivariate discrete data encoding within continuous optimization and learning architectures.

Overall, the proposed method offers a new perspective on integrating discrete symbolic logic into differentiable and smooth analytical systems.

\begin{center}
\emph{Oscillation decays, but structure persists.}
\end{center}

% \bibliographystyle{plain}
% \bibliography{references}

\begin{thebibliography}{1}

\bibitem{baydin2018automatic}
At{\i}l{\i}m~G{\"u}ne{\c{s}} Baydin, Barak~A. Pearlmutter, Alexey~Andreyevich Radul, and Jeffrey~Mark Siskind.
\newblock Automatic differentiation in machine learning: a survey.
\newblock {\em Journal of Machine Learning Research}, 18(153):1--43, 2018.

\bibitem{bengio2020discrete}
Yoshua Bengio, Andrea Lodi, and Antoine Prouvost.
\newblock Discrete optimization through continuous relaxations.
\newblock In {\em Proceedings of the International Conference on Machine Learning (ICML)}, 2020.

\bibitem{deboor2001splines}
Carl de~Boor.
\newblock {\em A Practical Guide to Splines}.
\newblock Springer, revised edition, 2001.

\bibitem{fasshauer2007meshfree}
Gregory~E. Fasshauer.
\newblock {\em Meshfree Approximation Methods with MATLAB}.
\newblock World Scientific, 2007.

\bibitem{jang2016gumbel}
Eric Jang, Shixiang Gu, and Ben Poole.
\newblock Categorical reparameterization with gumbel-softmax.
\newblock {\em arXiv preprint arXiv:1611.01144}, 2016.

\bibitem{krantz2002primer}
Steven~G. Krantz and Harold~R. Parks.
\newblock {\em A Primer of Real Analytic Functions}.
\newblock Birkh{\"a}user, 2nd edition, 2002.

\bibitem{rudin1976principles}
Walter Rudin.
\newblock {\em Principles of Mathematical Analysis}.
\newblock McGraw-Hill, 3rd edition, 1976.

\end{thebibliography}

\newpage
\appendix

\section{Code for Plotting the Smooth Counter Function}

The following Python code snippet generates the plot of the smooth counter function \( f_N(t) \) for \( N=5 \) and \( \delta=0.2 \), as shown in Figure~\ref{fig:fn_plot}.

\begin{verbatim}
import numpy as np
import matplotlib.pyplot as plt

# Parameters
delta = 0.2
N = 5
t = np.linspace(0, N + 3, 1000)  # Extra margin for better visualization

# Coefficient sequence a_n
def a_n(n):
    return (0.5)**n / n + (-1)**n / n

# Smooth counter function f_N(t)
def f_N(t, N, delta):
    result = np.zeros_like(t)
    for n in range(1, N + 1):
        result += a_n(n) * np.exp(-(t - n)**2 / (2 * delta**2))
    return result

# Evaluate f_N(t)
f_values = f_N(t, N, delta)

# Plot f_N(t)
plt.figure(figsize=(8,5))
plt.plot(t, f_values)
plt.title(r'Smooth Counter Function $f_N(t)$ for $N=5$, $\delta=0.2$')
plt.xlabel('t')
plt.ylabel('f_N(t)')
plt.grid(True)
plt.savefig('fn_plot_example.png', dpi=300)  # Save the plot
plt.show()
\end{verbatim}

This code constructs the smooth counter function \( f_N(t) \) by summing five localized Gaussian bumps with alternating-decaying coefficients and plots it over a suitable domain.

\section{Code for Generating the Integral Map Plot}

The following Python code snippet generates the plot of the integral map \( I(N) \) versus \( N \), illustrating the oscillatory decay toward zero, as shown in Figure~\ref{fig:i_plot}.

\begin{verbatim}
import numpy as np
import matplotlib.pyplot as plt

# Parameters
delta = 0.2
N_max = 20
t = np.linspace(0, N_max + 2, 1000)

# Coefficient sequence a_n
def a_n(n):
    return (0.5)**n / n + (-1)**n / n

# Smooth counter function f_N(t)
def f_N(t, N, delta):
    result = np.zeros_like(t)
    for n in range(1, N + 1):
        result += a_n(n) * np.exp(-(t - n)**2 / (2 * delta**2))
    return result

# Compute I(N) for N = 1 to N_max
I_N = []
for N in range(1, N_max + 1):
    f_t = f_N(t, N, delta)
    integral = np.trapz(f_t, t)
    I_N.append(integral)

# Plot I(N)
plt.figure(figsize=(8,5))
plt.plot(range(1, N_max + 1), I_N, marker='o')
plt.title('Integral Map I(N) vs N')
plt.xlabel('N')
plt.ylabel('I(N)')
plt.grid(True)
plt.savefig('i_plot_example.png', dpi=300)  # Save the plot
plt.show()
\end{verbatim}

This code constructs the function \( f_N(t) \) by summing smooth Gaussian bumps with alternating-decaying coefficients, computes the cumulative integral \( I(N) \), and plots its oscillatory behavior across \( N \).

\section{Code for Plotting Partial Sums of Coefficients}

The following Python code snippet generates the plot of partial sums
\[
S(N) := \sum_{n=1}^{N} a_n,
\]
up to \( N=20 \), illustrating structured oscillations around zero, as shown in Figure~\ref{fig:s_partial_sum}.

\begin{verbatim}
import numpy as np
import matplotlib.pyplot as plt

# Parameters
N_max = 20

# Coefficient sequence a_n
def a_n(n):
    return (0.5)**n / n + (-1)**n / n

# Compute partial sums S(N)
S_N = []
current_sum = 0
for n in range(1, N_max + 1):
    current_sum += a_n(n)
    S_N.append(current_sum)

# Plot S(N)
plt.figure(figsize=(8,5))
plt.plot(range(1, N_max + 1), S_N, marker='o')
plt.title(r'Partial Sums $S(N) = \sum_{n=1}^N a_n$ up to $N=20$')
plt.xlabel(r'$N$')
plt.ylabel(r'$S(N)$')
plt.grid(True)
plt.axhline(0, color='gray', linestyle='--')  # Zero line for reference
plt.savefig('partial_sum_plot.png', dpi=300)  # Save the plot
plt.show()
\end{verbatim}

This code computes and visualizes the partial sums of the alternating-decaying sequence \( a_n \), confirming the controlled oscillatory behavior and convergence toward zero.

\section{Code for Fully Smooth Integral Map Visualization}

The following Python code snippet generates the fully smooth version of the integral map \( I(N) \), based on the smoothing approach.

The code uses a smooth transition function \(\sigma(x)\) (a sigmoid) to interpolate the contributions of the Gaussian bumps and produces a globally \( C^\infty \) integral map suitable for differentiable applications.

\begin{verbatim}
import numpy as np
import matplotlib.pyplot as plt

# Parameters
delta = 0.2
N_max = 10
N_fine = np.linspace(0, N_max, 1000)  # Fine grid for smooth plot

# Coefficient sequence a_n
def a_n(n):
    return (0.5)**n / n + (-1)**n / n

# Smooth transition function sigma
def sigma(x, sharpness=10):
    return 1 / (1 + np.exp(sharpness * x))  # Sigmoid transition

# Fully smooth I(N)
def I_smooth(N_array, delta, sharpness=10):
    I_vals = np.zeros_like(N_array)
    n_max = int(N_max) + 10  # Extra terms for convergence
    n_values = np.arange(1, n_max + 1)
    a_values = np.array([a_n(n) for n in n_values])
    for i, N in enumerate(N_array):
        sigma_values = sigma(n_values - N, sharpness=sharpness)
        I_vals[i] = (
            delta * np.sqrt(2 * np.pi)
            * np.sum(sigma_values * a_values)
        )
    return I_vals

# Compute the smooth integral map
I_vals_smooth = I_smooth(N_fine, delta)

# Plotting
plt.figure(figsize=(8,5))
plt.plot(N_fine, I_vals_smooth)
plt.title(r'Fully Smooth Integral Map $I(N)$ for $N\in[0,10]$')
plt.xlabel(r'$N$')
plt.ylabel(r'$I(N)$')
plt.grid(True)
plt.savefig('fully_smooth_i_plot_example.png', dpi=300)
plt.show()
\end{verbatim}

The parameter \texttt{sharpness} controls how rapidly the contributions transition from active to inactive around each bump center.  
Higher values of \texttt{sharpness} approach the behavior of the piecewise-linear model, while moderate values produce fully smooth oscillations with soft transitions between integer contributions.

\section{Code for Table-Based Integer Recovery}

The following Python code snippet generates the discrete integral table \( (N, I(N)) \) to perform integer recovery.

\subsection{Generation of the Integral Table}

The following code computes \( I(N) \) for integers \( N = 1, \dotsc, N_{\text{max}} \), using the smooth counter construction:

\begin{verbatim}
import numpy as np

# Parameters
delta = 0.2
N_max = 30

# Coefficient sequence a_n
def a_n(n):
    return (0.5)**n / n + (-1)**n / n

# Precompute I(N)
N_values = np.arange(1, N_max + 1)
I_values = np.array([
    np.sqrt(2 * np.pi) * delta * sum(a_n(k) for k in range(1, n + 1))
    for n in N_values
])

# Table (N, I(N))
table = list(zip(N_values, I_values))
\end{verbatim}

This generates a list \texttt{table} containing pairs \((N, I(N))\), which can be used for numerical recovery.

\subsection{Recovery Procedure Using the Table}

Given an observed integral \( I^* \), the recovery procedure searches for the smallest \( N \) such that \( |I(N) - I^*| \) falls below a specified threshold:

\begin{verbatim}
# Observed integral
I_star = 0.028
epsilon = 0.005

# Find the minimal N such that |I(N) - I_star| < epsilon
recovered_N = None
for N, I_N in table:
    if abs(I_N - I_star) < epsilon:
        recovered_N = N
        break

print(f"Recovered N: {recovered_N}")
\end{verbatim}

In the example shown, the search identifies \( N = 8 \) as the integer corresponding to the given integral value \( I^* \approx 0.028 \).

\subsection{Remarks}

\begin{itemize}
    \item For more precise recovery or continuous interpolation, spline fitting over \((N, I(N))\) can be applied.
    \item For larger \( N \), binary search over sorted \((N, I(N))\) values improves efficiency.
\end{itemize}

\end{document}